\documentclass[12pt,amsfonts]{article}
\usepackage[margin=1in]{geometry}  
\usepackage{graphicx}              
\usepackage[latin1]{inputenc}
\usepackage{amsmath, amssymb, amsthm, epsfig}               
\usepackage{enumerate}
\usepackage{amsfonts}              
\usepackage{amsthm}                
\usepackage{amsthm}
\usepackage{enumerate}
\theoremstyle{plain}
\usepackage{color}
\def\dis{\displaystyle}
\def\nd{\noindent}
\def\thend{\rule{3mm}{3mm}}

\newtheorem{theorem}{Theorem}[section]

\newtheorem{lemma}{Lemma}[section]

\newtheorem{definition}{Definition}[section]
\newtheorem{corollary}{Corollary}[section]
\newtheorem*{theorem*}{Theorem}

\numberwithin{equation}{section}

\begin{document}
\title{Existence of a positive solution for a logarithmic Schr\"{o}dinger  equation with saddle-like potential}
\author{ Claudianor O. Alves\footnote{C.O. Alves was partially supported by CNPq/Brazil  304804/2017-7.} \,\, and \,\, Chao Ji \footnote{C. Ji was partially supported by Shanghai Natural Science Foundation(18ZR1409100).}}

\maketitle

\begin{abstract}
In this  article we use the variational method developed by Szulkin \cite{szulkin} to prove  the existence of a positive solution for the following logarithmic Schr\"{o}dinger equation
$$
\left\{
\begin{array}{lc}
-{\epsilon}^2\Delta u+ V(x)u=u \log u^2, & \mbox{in} \,\, \mathbb{R}^{N}, \\
u \in H^1(\mathbb{R}^{N}), & \;  \\
\end{array}
\right.
$$
where  $\epsilon >0, N \geq 1$  and $V$ is a saddle-like potential.
\end{abstract}



\section{Introduction}
 In recent years, the nonlinear Schr\"odinger equation
\begin{equation}\label{NLS}
i \epsilon \frac{\partial  \Psi}{\partial t} = - \epsilon^{2}\Delta \Psi + (V(x	)+E)\Psi- f(\Psi), \quad\text{for}\,\, z\in \mathbb{R}^{N}
\end{equation}
 where $N\geq 2$, $\epsilon>0$, $V, f$ are continuous functions, has been studied by many researchers.  It is very important
 to seek the standing wave solutions of (\ref{NLS}), which are solutions of the type $\Phi(z, t)=\exp(iEt)u(z), u:\mathbb{R}^{N}\rightarrow \mathbb{R}$, where $u$ is a solution of the elliptic equation
$$
\left\{
\begin{array}{lc}
-{\epsilon}^2\Delta u+ V(x)u=f(u), & \mbox{in} \,\, \mathbb{R}^{N}, \\
u \in H^1(\mathbb{R}^{N}), & \;  \\
\end{array}
\right.
\eqno{(P_\epsilon)}
$$
The existence and concentration of positive solutions for general semilinear elliptic equation $(P_\epsilon)$, for the case $N\geq 2$, have been extensively studied, we refer to \cite{Alves, AlvesMiyagaki, BPW, BW, DF, DFM, fw, o, r, W} for the advances in this area.

In \cite{DFM}, del Pino, Felmer and Miyagaki considered the case where potential $V$ has a geometry like saddle, essentially
they assumed the following conditions on $V$: First of all, they fixed  two subspaces $X$, $Y\subset \mathbb{R}^{N}$ such that
$\mathbb{R}^{N}=X\oplus Y$. By supposing that $V$ is bounded, they fixed $c_{0}$, $c_{1}>0$ satisfying
$$
c_{0}=\inf_{z\in \mathbb{R}^{N}}V(z)>0 \quad\text{and}\quad c_{1}=\sup_{x\in X}V(x).
$$

Furthermore, they also assumed that $V\in C^{2}(\mathbb{R}^{N})$ and it verifies the  following geometric conditions:
\begin{itemize}
	\item[\rm ($V1$)]
	$$
	c_{0}=\inf_{R>0}\sup_{x\in \partial B_{R}(0)\cap X}V(x)<\inf_{y\in Y_\lambda}V(y),
	$$
for some $\lambda \in (0,1)$ and $Y_\lambda=\{z \in \mathbb{R}^N\,:\, |z.y| > \lambda|z||y|, \, \mbox{for some} \,\, y \in Y\}$ 	
	
	\item[\rm ($V2$)]The functions $V$, $\frac{\partial V}{\partial x_{i}}$ and $\frac{\partial^{2} V}{\partial x_{i}\partial x_{j}}$ are bounded
in $\mathbb{R}^N$ for all $i, j\in \{1, \cdots, N\}$.

		\item[\rm ($V3$)] $V$ satisfies the Palais-Smale condition, that is, if $(x_{n})\subset \mathbb{R}^N$ is a sequence such that
$(V(x_{n}))$ is convergent and $\nabla V(x_{n})\rightarrow 0$, then $(x_{n})$ possesses a convergent subsequence in $\mathbb{R}^N$.

\end{itemize}

Using the above conditions on $V$, and supposing that
$$
c_{1}<2^{\frac{2(p-1)}{N+2-p(N-2)}}c_{0},
$$
del Pino, Felmer and Miyagaki used the variational method to show the existence of positive solutions for the following problem
$$
-\epsilon^{2}\Delta u+ V(z)u=\vert u\vert^{p-2}u,\, \mbox{in} \,\, \mathbb{R}^{N},
$$
where $p\in (2, 2^{*})$ if $N\geq 3$ and $p\in (2, +\infty)$ if $N=1, 2$, for $\epsilon>0$
small enough.

Under the same assumptions on the potential $V$,  Alves \cite{Alves} considered the existence of a positive solution for the following elliptic
equation with exponential critical growth in $\mathbb{R}^{2}$
$$
-\epsilon^{2}\Delta u+ V(x)u=f(u),  \,\, \mbox{in} \,\,\mathbb{R}^{2}.
$$
After, in \cite{AlvesMiyagaki}, Alves and Miyagaki  studied the following critical nonlinear elliptic equation with saddle-like potential in $\mathbb{R}^{N}$
$$
\epsilon^{2s}(-\Delta)^{s} u+ V(z)u=\lambda \vert u\vert^{q-2}u+\vert u\vert^{2_{s}^{*}-2}u,\,\, \mbox{in} \,\, \mathbb{R}^{N},
$$
where $\epsilon, \lambda>0$ are positive parameters, $q\in (2, 2_{s}^{*})$, $2_{s}^{*}=\frac{2N}{N-2s}$, $N>2s$, $s\in (0, 1)$, $(-\Delta)^{s} $
is fractional Laplacian. Under similar assumptions on $V$,  the authors obtained the existence of a positive solution by using the variational method.

Recently, the logarithmic Schr\"odinger equation given by
$$
i \epsilon \partial_t \Psi = - \epsilon^{2}\Delta \Psi + (W(x	)+w)\Psi- \Psi \log|\Psi|^2, \ \Psi:[0, \infty) \times \mathbb{R}^{N} \rightarrow \mathbb{C}, \ N \geq 1,
$$
 has  also been received considerable attention. This equation appears in interesting physical applications, such as quantum mechanics, quantum optics, nuclear physics, transport and diffusion phenomena, open quantum systems, effective quantum gravity, theory of superfluidity and Bose-Einstein condensation (see \cite{z} and the references therein). In its turn, standing waves solution, $\Psi$, for this logarithmic Schr\"odinger equation  is related to solutions of the equation
$$
-\epsilon^{2}\Delta u+ V(x)u=u \log u^2,  \,\, \mbox{in} \,\, \mathbb{R}^{N}.
$$
Besides the importance in applications, this last  equation also raises many difficult mathematical problems.  The natural candidate for the associated energy functional would formally be the functional
\begin{equation}\label{lula}
\widehat{I}_\epsilon(u)=\frac{1}{2}\displaystyle \int_{\mathbb{R}^N}(\epsilon^{2}|\nabla u|^2+(V(x)+1)|u|^2)dx-\displaystyle \frac{1}{2}\int_{\mathbb{R}^N} u^2 \log u^2\,dx.
\end{equation}
It is easy to see that each critical point of $\widehat{I}_\epsilon$ is a solution of (\ref{lula}).
However, this functional is not well defined in $H^{1}(\mathbb{R}^N)$ because there is $u \in H^{1}(\mathbb{R}^N)$ such that $\int_{\mathbb{R}^N}u^{2}\log u^2 \, dx=-\infty$. In order to overcome this technical difficulty some authors have used different techniques, for more details see \cite{AlvesdeMorais}, \cite{AlvesdeMoraisFigueiredo}, \cite{AlvesChao},  \cite{squassina}, \cite{ASZ}, \cite{DZ}, \cite{cs},  \cite{sz}, \cite{sz2}, \cite{TZ}  and their references.

Motivated by studies found in the above-mentioned papers and the results  obtained in \cite{Alves, AlvesMiyagaki,DFM}, in the present paper, our main goal is to show the existence of a positive solution for the following logarithmic Schr\"{o}dinger  equation
\begin{equation}\label{1}
-\epsilon^{2}\Delta u+ V(x)u=u \log u^2,  \quad \mbox{in} \quad \mathbb{R}^{N},
\end{equation}
where $N\geq 1$, $\epsilon >0$ is a positive parameter and the potential $V$  satisfies
the conditions $(V1)-(V3)$. In our case, different from \cite{DFM}, we can assume $c_0>-1$.

By a change of variable, we know that problem (\ref{1}) is equivalent to the problem
\begin{equation}\label{1'}
\left\{
\begin{array}{lc}
-\Delta u+ V(\epsilon x)u=u \log u^2, & \mbox{in} \quad \mathbb{R}^{N}, \\
u \in H^1(\mathbb{R}^{N}). & \;  \\
\end{array}
\right.
\end{equation}
We shall use the variational method found in Szulkin \cite{szulkin} to study the problem (\ref{1'}). Firstly notice that, a weak solution of (\ref{1'}) in $H^1(\mathbb{R}^{N})$ is a critical point of the associated energy functional
$$
J_{\epsilon}(u)=\dis\frac{1}{2}\int_{\mathbb{R}^N} \big (|\nabla u|^2+(V(\epsilon x) +1)|u|^2\big)dx-\dis\frac{1}{2}\int_{\mathbb{R}^N} u^2\log u^2dx.
$$
\begin{definition}
For us, a positive solution of (\ref{1'}) means a positive function $u \in H^1(\mathbb{R}^{N}) \setminus \{0\}$ such that $ u^2\log u^2 \in L^1(\mathbb{R}^{N})$ and
\begin{equation}
\displaystyle \int_{\mathbb{R}^N} (\nabla u \cdot \nabla v +V(\epsilon x)u \cdot v)dx=\displaystyle \int_{\mathbb{R}^N} uv \log u^2dx,\,\, \mbox{for all }\, v \in C^\infty_0(\mathbb{R}^{N}).
\end{equation}
\end{definition}

Following the approach explored in \cite{AlvesdeMorais,cs,sz}, due to the lack of smoothness of $J_{\epsilon}$, let us decompose it into a sum of a $C^1$ functional plus a convex lower semicontinuous functional, respectively. For $\delta>0$, let us define the following functions:

$$
F_1(s)=
\left\{ \begin{array}{lc}
0, & \; s=0 \\
-\frac{1}{2}s^2\log s^2 & \; 0<|s|<\delta \\
-\frac{1}{2}s^2(
\log\delta^2+3)+2\delta|s|-\frac{1}{2}\delta^2,  & \; |s| \geq \delta
\end{array} \right.
$$

\noindent and

$$
F_2(s)=
\left\{ \begin{array}{lc}
0, & \;  |s|<\delta \\
\frac{1}{2}s^2\log(s^2/\delta^2)+2\delta|s|-\frac{3}{2}s^2-\frac{1}{2}\delta^2,  & \; |s| \geq  \delta.
\end{array} \right
.$$
Therefore
\begin{equation}\label{efes}
F_2(s)-F_1(s)=\frac{1}{2}s^2\log s^2, \quad \forall s \in \mathbb{R},
\end{equation}
and the functional $J_\epsilon:H^{1}(\mathbb{R}^N) \rightarrow (-\infty, +\infty]$ may be rewritten as

\begin{equation}\label{funcional}
J_\epsilon(u)=\Phi_\epsilon(u)+\Psi(u), \quad u \in H^{1}(\mathbb{R}^N)
\end{equation}
 where
\begin{equation}\label{fi}
\Phi_\epsilon(u)=\frac{1}{2}\int_{\mathbb{R}^N} (|\nabla u|^{2}+(V(\epsilon x)+1))|u|^{2})\,dx-\displaystyle \int_{\mathbb{R}^N} F_2(u)\,dx,
\end{equation}
\noindent and
\begin{equation}\label{psi}
\Psi(u)=\displaystyle \int_{\mathbb{R}^N} F_1(u)\,dx.
\end{equation}

It was proved in \cite{cs} and \cite{sz} that $F_1$ and $F_2$ verify the following properties:
\begin{equation}\label{eq1}
F_1, F_2 \in C^1(\mathbb{R},\mathbb{R}).
\end{equation}
If $\delta >0$ is small enough, $F_1$ is convex, even, $F_1(s)\geq 0$ for all $ s\in \mathbb{R}$ and
\begin{equation}\label{eq2}
F'_1(s)s\geq 0, \ s \in \mathbb{R}.
\end{equation}
For each fixed $p \in (2, 2^*)$, there is $C>0$ such that
\begin{equation}\label{eq5}
|F'_2(s)|\leq C|s|^{p-1}, \quad \forall s \in \mathbb{R}.
\end{equation}

Using the above information, it follows that $\Phi_\epsilon \in C^{1}(H^1(\mathbb{R}^{N}),\mathbb{R})$, $\Psi$ is convex and lower semicontinuous, but $\Psi$ is not a $C^1$ functional, since we are working on $\mathbb{R}^N$.

Before to state our main result, we need to fix some notations. If potential $V$ in (\ref{1'})  is replaced by a constant $A>-1$, we have the following problem
\begin{equation}\label{constant}
\left\{
\begin{array}{lc}
-\Delta u+ A u=u \log u^2, & \quad \mbox{in} \quad \mathbb{R}^{N}, \\
u \in H^1(\mathbb{R}^{N}). & \;  \\
\end{array}
\right.
\end{equation}
The corresponding energy functional associated to (\ref{constant}) will be denoted by \linebreak $J_A:H^{1}(\mathbb{R}^N) \rightarrow (-\infty, +\infty]$ and defined as
$$
J_{A}(u)=\dis\frac{1}{2}\int_{\mathbb{R}^N} \big (|\nabla u|^2+({A} +1)|u|^2\big)dx-\dis\frac{1}{2}\int_{\mathbb{R}^N} u^2\log u^2dx.
$$

In \cite{sz} is proved that problem (\ref{constant}) has a positive ground state solution given by
\begin{equation}\label{sisi}
m(A):=\inf_{u \in \mathcal{N}_{A}}J_{A}(u)=\inf_{u \in D(J_{A})\backslash \{0\}}\Big\{\max_{t\geq 0}J_{A}(tu)\Big\},
\end{equation}
where
$$
\mathcal{N}_{A}=\left\{u \in D(J_{A})\backslash \{0\}; J_{A}(u)=\displaystyle \frac{1}{2} \int_{\mathbb{R}^N} |u|^{2}\,dx  \right\}
$$
and
$$
D(J_{A})=\left\{u \in H^{1}(\mathbb{R}^N)\,:\, J_{A}(u)<+\infty \right\}.
$$

The main result of this paper is the following:

\begin{theorem}\label{teorema}
Suppose that $V$ satisfies $(V1)-(V3)$. If
 \begin{itemize}
	\item[\rm ($V4$)]
	$$
	m(V(0))\geq 2m(c_{0})\quad \text{and}\quad c_{1}\leq c_{0}+\frac{3}{10}c_{2},
	$$
\end{itemize}
 where $c_{2}=\min\{1, c_{0}\}$, then there is $\epsilon_*>0$ such that equation (\ref{1}) has a positive solution for all $\epsilon \in (0, \epsilon_*)$.
\end{theorem}

Theorem \ref{teorema} is inspired from \cite{Alves, AlvesMiyagaki,DFM}, however we are working with the logarithmic Schr\"{o}dinger equation, whose the energy functional associated is not continuous, for this reason,  some estimates  for this problem are also very delicate and different from those used in the Schr\"{o}dinger equation (\ref{NLS}). The reader is invited to see that the way how we apply the deformation lemma in Section 3 is different of that explored in \cite{Alves, AlvesMiyagaki,DFM}, because in our paper the Nehari set is not a manifold, and so, new arguments have been developed to overcome this difficulty. The plan of the paper is as follows: In Section 2 we show some estimates that do not appear in \cite{Alves, AlvesMiyagaki,DFM} and prove a technical result. In Section 3  we apply the deformation lemma to prove Theorem 1.1.

\vspace{0.5 cm}

\noindent \textbf{Notation:} From now on in this paper, otherwise mentioned, we use the following notations:
\begin{itemize}
	\item $B_r(u)$ is an open ball centered at $u$ with radius $r>0$, $B_r=B_r(0)$.
	
	\item If $g$ is a measurable function, the integral $\int_{\mathbb{R}^N}g(z)dz$ will be denoted by $\int g(z)dz$.
	
	\item   $C$ denotes any positive constant, whose value is not relevant.
	
	\item  $|\,\,\,|_p$ denotes the usual norm of the Lebesgue space $L^{p}(\mathbb{R}^N)$, for $p \in [1,+\infty]$.
	
	\item  $H_c^{1}(\mathbb{R}^N)=\{u \in H^{1}(\mathbb{R}^N)\,:\, u \,\, \mbox{has compact support}\, \}.$
	
	\item $o_{n}(1)$ denotes a real sequence with $o_{n}(1)\to 0$ as $n \to +\infty$.
	
	\item $2^*=\frac{2N}{N-2}$ if $N \geq 3$ and $2^*=+\infty$ if $N=1,2$.
\end{itemize}

\section{Technical results}\label{tpm}

We begin this section recalling some definitions that can be found in \cite{szulkin}.

\begin{definition} \label{definicao}
	
	Let $E$ be a Banach space, $E'$ be the dual space of $E$ and $\langle \cdot,\cdot \rangle$ be the duality paring between $E'$ and $E$. Let $J:E \to \mathbb{R}$ be a functional of the form $J(u)=\Phi(u)+\Psi(u)$, where $\Phi \in C^{1}(E,\mathbb{R})$ and $\Psi$ is convex and lower semicontinuous. Let us list some definitions:

	\noindent (i) The sub-differential $\partial J(u)$ of the functional $J$ at a point $u \in E$ is the following set
	
	\begin{equation}\label{subdif}
	\{w \in E': \langle \Phi'(u), v-u \rangle+\Psi(v)-\Psi(u)\geq \langle w, v-u\rangle, \ \forall v \in E\}.
	\end{equation}
	
	\noindent (ii) A critical point of $J$ is a point $u \in E$ such that $J(u)< + \infty$ and
	$0 \in \partial J(u)$,\textit{ i.e.}
	
	\begin{equation}\label{critical}
	\langle \Phi'(u), v-u\rangle+\Psi(v)-\Psi(u)\geq 0, \ \forall v \in E.
	\end{equation}
	
	\noindent (iii) A Palais-Smale sequence at level $d$ for $J$ is a sequence $(u_n)\subset E$ such that $J(u_n)\rightarrow d$ and there is a numerical sequence $\tau_n \rightarrow 0^+$ with
	\begin{equation}\label{psequence}
	\langle \Phi'(u_n), v-u_n\rangle+\Psi(v)-\Psi(u_n)\geq -\tau_n||v-u_n||, \ \forall v \in E.
	\end{equation}

	\noindent (iv) The functional $J$ satisfies the Palais-Smale condition at level $d$ $((PS)_d$ condition, for short$)$ if all Palais-Smale sequences at level $d$ has a convergent subsequence.
	
	\noindent (v) The effective domain of $J$ is the set $D(J)=\{u \in E: J(u)< +\infty\}.$
	
\end{definition}
To proceed further we gather and state below some useful results that leads to a better understanding of the problem and of its particularities. In what follows, for each $u \in D(J_\epsilon)$, we set the functional $J'_\epsilon(u):H_c^{1}(\mathbb{R}^N) \to \mathbb{R}$ given by
$$
\langle J'_\epsilon(u),z\rangle=\langle \Phi_\epsilon'(u),z\rangle-\int F'_1(u)z\,dx, \quad \forall z \in H_c^{1}(\mathbb{R}^N)
$$
and define
$$
\|J'_\epsilon(u)\|=\sup\left\{\langle J'_\epsilon(u),z\rangle\,:\, z \in H_c^{1}(\mathbb{R}^N) \quad \mbox{and} \quad \|z\|_\epsilon \leq 1 \right\}.
$$
If $\|J'_\epsilon(u)\|$ is finite, then $J_\epsilon'(u)$ may be extended to a bounded operator in $ H^{1}(\mathbb{R}^N)$, and so,  it can be seen as an element of $(H^{1}(\mathbb{R}^N))'$.

\begin{lemma} \label{lema} Let $J_\epsilon$ satisfy  (\ref{funcional}), then:\\
	\noindent (i) If $u \in D(J_{\epsilon})$ is a critical point of $J_\epsilon$, then
	$$
	\langle \Phi_\epsilon'(u), v-u \rangle +\Psi(v)-\Psi(u)\geq 0, \quad  \forall v \in H^{1}(\mathbb{R}^N),
	$$
	or equivalently
	$$
	\int \nabla u \nabla (v-u)\, dx + \int (V(\epsilon x)+1)u(v-u)\,dx + \int F_1(v)\, dx - \int F_1(u)\, dx \geq \int F'_2(u)(v-u)\,dx,  \forall v \in H^{1}(\mathbb{R}^N).
	$$
	\noindent (ii) For each $u \in D(J_\epsilon)$ such that $\|J'_\epsilon(u)\|< +\infty$, we have $\partial J_\epsilon(u) \not= \emptyset$, that is, there is $w \in  (H^{1}(\mathbb{R}^N))'$, which is denoted by $w=J'_\epsilon(u)$, such that
	$$
	\langle \Phi_\epsilon'(u), v-u \rangle + \int F_1(v)\, dx - \int F_1(u)\, dx \geq \langle w,v-u \rangle, \quad  \forall v \in H^{1}(\mathbb{R}^N), \,\, ( \mbox{see} \, \cite{sz2,TZ} )
	$$

	\noindent 	(iii) If a function $u \in D(J_\epsilon)$ is a critical point of $J_\epsilon$, then $u$ is a solution of (\ref{1'}) ( (i) in Lemma 2.4, \cite{cs}).
	
	\noindent 	(iv) If  $(u_n)\subset H^{1}(\mathbb{R}^N)$ is a Palais-Smale sequence, then
	\begin{equation} \label{ps}
	\langle J'_\epsilon(u_n), z \rangle=o_n(1)\|z\|_\epsilon , \quad \forall z \in H_c^{1}(\mathbb{R}^N).
	\end{equation}
	
	(see (ii) in Lemma 2.4, \cite{cs}).
	
	\noindent	(v) If $\Omega$ is a bounded domain with regular boundary, then $\Psi$ (and hence $J_\epsilon$) is of class $C^1$ in $H^{1}(\Omega)$ (Lemma 2.2 in \cite{sz}). More precisely, the functional
	$$
	\Psi(u)=\int_{\Omega}F_1(u)\,dx, \quad \forall u \in H^{1}(\Omega)
	$$
	belongs to $C^{1}(H^{1}(\Omega),\mathbb{R})$. 	
\end{lemma}

As a consequence of the above proprieties, we have  the following results whose the proofs can be found in \cite{AlvesdeMorais}.
\begin{lemma}
	If $u \in D(J_\epsilon)$ and $\|J'_\epsilon(u)\|< +\infty$, then $F^{'}_1(u)u \in L^1(\mathbb{R}^{N})$.
\end{lemma}

An immediate consequence of the last lemma is the following.

\begin{corollary} \label{C1} For each $u \in D(J_\epsilon) \setminus \{0\}$ with $\|J'_\epsilon(u)\|< +\infty$, we have that
	$$
	J'_\epsilon(u)u=\int (|\nabla u|^{2}+V(\epsilon x)|u|^2)\,dx-\int u^{2} \log u^2\,dx
	$$
	and
	$$
	J_{\epsilon}(u)-\frac{1}{2}J'_{\epsilon}(u)u=\frac{1}{2}\int |u|^{2}\, dx.
	$$

\end{corollary}

\begin{corollary} \label{C2} If $(u_n) \subset H^{1}(\mathbb{R}^N)$ is a $(PS)$ sequence for $J_\epsilon$, then $J_\epsilon'(u_n)u_n=o_n(1)\|u_n\|_\epsilon$. If $(u_n)$ is bounded, we have
	$$
	J_{\epsilon}(u_n)=J_{\epsilon}(u_n)-\frac{1}{2}J_{\epsilon}'(u_n)u_n+o_n(1)\|u_n\|_\epsilon=\frac{1}{2}\int |u_n|^{2}\, dx+o_n(1)\|u_n\|_\epsilon, \quad \forall n \in \mathbb{N}.
	$$
\end{corollary}

\begin{corollary} \label{C3} If $u \in H^{1}(\mathbb{R}^N)$ is a critical point of $J_\epsilon$ and $v \in H^{1}(\mathbb{R}^N)$ verifies $F'_1(u)v \in L^{1}(\mathbb{R}^N)$, then $J'_\epsilon(u)v=0$.
	
\end{corollary}

Next we will prove some results that will be useful in the proof of Theorem  \ref{teorema}.

\begin{lemma}\label{boundedness}
For any $\epsilon>0$, all $(PS)$ sequences of  $J_{\epsilon}$ are bounded in $H^{1}(\mathbb{R}^N)$.
\end{lemma}
\begin{proof}
Let $(u_{n})$ be a $(PS)_{d}$ sequence. Then
$$
\int |u_n|^2dx=2J_{\epsilon}(u_n)-J'_{\epsilon}(u_n)u_n=2d+o_n(1)+o_n(1)\|u_n\| \leq C+o_n(1)\|u_n\|,
$$
for some $C>0$. Consequently
\begin{equation}\label{rimsky}
|u_n|^2_{2}\leq  C+o_n(1)\|u_n\|.
\end{equation}
Now, let us employ the following logarithmic Sobolev inequality found in \cite{lieb},
\begin{equation}\label{xerazade0}
\int u^2 \log u^2dx \leq\frac{a^2}{\pi}|\nabla u|^2_{2}+ \big( \log |u|^2_{2}-N(1+\log a)\big)|u|^2_{2}
\end{equation}
for all $a>0$. Fixing $\frac{a^2}{\pi}=\frac{1}{4}$ and $\xi \in (0,1)$, the inequalities (\ref{rimsky}) and (\ref{xerazade0}) yield

\begin{eqnarray}
 \int u_n^2 \log u_n^2 dx  &\leq&\frac{1}{4}|\nabla u_n|_{2}^2+ C\big(  \log {|u_n|^2_{2}+1} \big)
{|u_n|^2_{2}}\nonumber\\
&\leq&\frac{1}{4}|\nabla u_n|^2_{2}+C_1\big(1+\|u_n\|\big)^{1+\xi} \label{isso}.
\end{eqnarray}
Since
$$
d+o_n(1)=J_{\epsilon}(u_n)=\frac{1}{2}|\nabla u_n|^2_{2}+  \int (V(\epsilon x)+1)|u_n|^2dx-\frac{1}{2}\int u_n^2 \log u_n^2dx
$$
assertion (\ref{isso}) assures that

\begin{equation}\label{porta}
d+o_n(1) \geq C \big[\| u_n\|^2-\big(1+\|u_n\|\big)^{1+\xi} \big],
\end{equation}
showing that the sequence $(u_n)$ is bounded in $H^{1}(\mathbb{R}^N)$.
\end{proof}
\begin{lemma}\label{le22}
Under the assumptions $(V1)-(V4)$, for each $\sigma>0$, there is $\epsilon_{0}=\epsilon_{0}(\sigma)>0$, such that
$J_{\epsilon}$ satisfies the $(PS)_{c}$ condition for all $c\in (m(c_{0})+\sigma/2, 2m(c_{0})-\sigma)$, for all $\epsilon\in (0, \epsilon_{0})$.
\end{lemma}

\begin{proof}
We shall prove the lemma arguing by contradiction, by supposing that there are $\sigma>0$ and $\epsilon_{n}\rightarrow 0$,
such that $J_{\epsilon_{n}}$ does not satisfy the $(PS)$ condition. Therefore, there is $c_{n}\in (m(c_{0})+\sigma/2, 2m(c_{0})-\sigma)$ such that $J_{\epsilon_{n}}$ does not satisfy the $(PS)_{c_{n}}$ condition.
Then, there exists a sequence $(u_{m}^{n})$ such that
 \begin{equation}\label{eq31}
\lim_{m\to +\infty}J_{\epsilon_{n}}(u_{m}^{n})=c_{n}\quad\text{and}\quad \lim_{m\to +\infty}\|J'_{\epsilon_{n}}(u_{m}^{n})\|=0.
\end{equation}
By Lemma \ref{boundedness}, the sequence $(u_{m}^{n})$ is bounded in $H^{1}(\mathbb{R}^N)$, then
 \begin{equation}\label{eq32}
u_{m}^{n}\rightharpoonup u_{n}\quad \text{in}\,\,H^{1}(\mathbb{R}^N)\quad\text{but}\quad u_{m}^{n}\not\rightarrow u_{n}\quad \text{in}\,\,H^{1}(\mathbb{R}^N).
\end{equation}
Now, we claim that there is $\delta>0$, such that
\begin{equation*}
\liminf_{m\to +\infty}\sup_{y\in \mathbb{R}^N}\int_{B_R(y)}\vert u_{m}^{n}\vert^{2}dx\geq \delta,\quad \forall n\in \mathbb{N}.
\end{equation*}
Indeed, on the contrary, there is $(n_{j})\subset \mathbb{N}$ satisfying
\begin{equation*}
\liminf_{m\to +\infty}\sup_{y\in \mathbb{R}^N}\int_{B_R(y)}\vert u_{m}^{n_{j}}\vert^{2}dx\leq \frac{1}{j},\quad \forall j \in \mathbb{N}.
\end{equation*}
Arguing as in Lions \cite{lions},
\begin{equation*}
\limsup_{m\to +\infty}\vert u_{m}^{n_{j}}\vert_{p}=o_{j}(1),\quad \forall p\in(2, 2^{*}).
\end{equation*}
Then, by (\ref{eq5}), we would get
\begin{equation*}
\limsup_{m\to +\infty}\int F'_{2}(u_{m}^{n_{j}})u_{m}^{n_{j}}dx=o_{j}(1).
\end{equation*}
On the other hand, by (\ref{eq31}), we obtain
\begin{equation*}
\Vert u_{m}^{n_{j}}\Vert_{\epsilon}^{2}+\int F'_{1}(u_{m}^{n_{j}})u_{m}^{n_{j}}dx=J'_{\epsilon_{n_{j}}}(u_{m}^{n_{j}})u_{m}^{n_{j}}+\int F'_{2}(u_{m}^{n_{j}})u_{m}^{n_{j}}dx=o_{m}(1)\Vert u_{m}^{n_{j}}\Vert_{\epsilon}+\int F'_{2}(u_{m}^{n_{j}})u_{m}^{n_{j}}dx,
\end{equation*}
and
\begin{equation*}
\limsup_{m\to +\infty} \Vert u_{m}^{n_{j}}\Vert_{\epsilon}^{2}=o_{j}(1).
\end{equation*}
It is easy to see that $u_{m}^{n_{j}}\not\rightarrow 0$ as $m\rightarrow\infty$ in $H^{1}(\mathbb{R}^N)$, otherwise we  get
$$
\int F'_{1}(u_{m}^{n_{j}})u_{m}^{n_{j}}dx \to 0,
$$
which together with convexity of $F_1$ yields
$$
\int F_{1}(u_{m}^{n_{j}})dx \to 0.
$$
This limit combined with $u_{m}^{n_{j}} \rightarrow 0$ implies that $J_{\epsilon_{n_{j}}}(u_{m}^{n_{j}})\rightarrow 0$, which contradicts with $c_{n}>0$.  Thus
\begin{equation*}
\liminf_{m\to +\infty} \Vert u_{m}^{n_{j}}\Vert_{\epsilon}^{2}>0.
\end{equation*}
Without loss of generality, we can assume that $\{u_{m}^{n_{j}}\}\subset H^{1}(\mathbb{R}^N)\setminus \{0\}$. Thereby, there exists
$t_{m}^{n_{j}}\in (0, +\infty)$ such that
\begin{equation*}
t_{m}^{n_{j}}u_{m}^{n_{j}}\in \mathcal{N}_{\epsilon_{n_{j}}}.
\end{equation*}
Since
\begin{equation}\label{eq21}
J_{\epsilon_{n_{j}}}(u_{m}^{n_{j}})=\frac{(2\log t_{m}^{n_{j}}+1)}{2}\int \vert u_{m}^{n_{j}}\vert^{2}dx
\end{equation}
and
\begin{equation}\label{eq210}
J_{\epsilon_{n_{j}}}(u_{m}^{n_{j}})=\frac{1}{2}\int \vert u_{m}^{n_{j}}\vert^{2}dx+o_m(1),
\end{equation}
it is easy to see that
\begin{equation*}
\lim_{m\to +\infty}t_{m}^{n_{j}}=1.
\end{equation*}
Recalling that
 \begin{equation*}
 J_{\epsilon_{n_{j}}}(t_{m}^{n_{j}}u_{m}^{n_{j}})=\frac{1}{2}\int \vert t_{m}^{n_{j}} u_{m}^{n_{j}}\vert^{2}dx,
\end{equation*}
we have that
 \begin{equation*}
\lim_{m\to +\infty}J_{\epsilon_{n_{j}}}(t_{m}^{n_{j}}u_{m}^{n_{j}})=\lim_{m\to +\infty}J_{\epsilon_{n_{j}}}(u_{m}^{n_{j}}).
\end{equation*}
From the above limits, there is $r_{m}^{n_{j}}\in (0, 1)$ such that
\begin{equation*}
r_{m}^{n_{j}}(t_{m}^{n_{j}}u_{m}^{n_{j}})\in \mathcal{N}_{c_{0}}.
\end{equation*}
Hence,
\begin{eqnarray}\label{eq33}
m(c_{0})&\leq& \limsup_{m\to +\infty}J_{c_{0}}(r_{m}^{n_{j}}(t_{m}^{n_{j}}u_{m}^{n_{j}}))\leq \limsup_{m\to +\infty}J_{\epsilon_{n_{j}}}(t_{m}^{n_{j}}u_{m}^{n_{j}})=\limsup_{m\to +\infty}J_{\epsilon_{n_{j}}}(u_{m}^{n_{j}})\nonumber\\
&=&\frac{1}{2}\limsup_{m\to +\infty}\int\vert u_{m}^{n_{j}}\vert^{2}dx\leq C \limsup_{m\to +\infty}\Vert u_{m}^{n_{j}}\Vert^{2},
\end{eqnarray}
that is
\begin{equation*}
0<m(c_{0})\leq o_{j}(1),
\end{equation*}
which is a contradiction.

From the above study, for each $m\in \mathbb{N}$, there is $m_{n}\in \mathbb{N}$ such that
\begin{equation*}
\int_{B_R(z_{m_{n}}^{n})}|u_{m_{n}}^{n}|^{2}\,dx\geq \frac{\delta}{2},\quad |\epsilon_{n} z_{m_{n}}^{n}|\geq n, \quad \Vert J'_{\epsilon_{n}}(u_{m_{n}}^{n})\Vert\leq \frac{1}{n}\,\, \text{and}\,\, \vert J_{\epsilon_{n}}(u_{m_{n}}^{n})-c_{n}\vert\leq \frac{1}{n}.
\end{equation*}
In what follows, we denote by $(z_{n})$ and $(u_{n})$ the sequences $(z_{m_{n}}^{n})$ and $(u_{m_{n}}^{n})$ respectively. Thus,
\begin{equation*}
\int_{B_R(z_{n})}|u_{n}|^{2}\,dx\geq \frac{\delta}{2},\quad |\epsilon_{n} z_{n}|\geq n, \quad \Vert J'_{\epsilon_{n}}(u_{n})\Vert\leq \frac{1}{n}\,\, \text{and}\,\, \vert J_{\epsilon_{n}}(u_{n})-c_{n}\vert\leq \frac{1}{n}.
\end{equation*}

Arguing as in the proof of Lemma \ref{boundedness}, the sequence $(u_{n})$ is a bounded in $H^{1}(\mathbb{R}^N)$.  Then, for some subsequence, there exists $u\in H^{1}(\mathbb{R}^N)$ such that
\begin{equation*}
u_{n}\rightharpoonup u\quad\text{in}\,\, H^{1}(\mathbb{R}^N).
\end{equation*}
We claim that $u=0$, because if $u\neq 0$, the limit $\Vert J'_{\epsilon_{n}}(u_{n})\Vert\rightarrow 0$ together with $(V2)$ yield
$u$ is a nontrivial solution of the problem
\begin{equation*}
-\Delta u+ V(0)u=u \log u^2,  \quad \mbox{in} \, \,\mathbb{R}^{N}.
\end{equation*}
Then, combining the definition of $m(V(0))$ with $(V4)$, we get
\begin{equation*}
J_{V(0)}(u)\geq m(V(0))\geq 2m(c_{0}).
\end{equation*}
On the other hand, using the Fatou's lemma, one has
$$
J_{V(0)}(u)\leq \frac{1}{2}\liminf_{n\to +\infty}\int|u_{n}|^2dx=\liminf_{n\to +\infty}(J_{\epsilon_{n}}(u_{n})-\frac{1}{2}J'_{\epsilon_{n}}(u_{n})u_{n})=
\liminf_{n\to +\infty}J_{\epsilon_{n}}(u_{n}),
$$
and so,
$$
J_{V(0)}(u)\leq \liminf_{n\to +\infty}c_{n}\leq 2m(c_{0})-\sigma,
$$
which is a contradiction, from where it follows that $u_{n}\rightharpoonup 0$ in $H^{1}(\mathbb{R}^N)$.

Considering $\omega_{n}=u_{n}(\cdot+z_{n})$, we have that $(\omega_{n})$ is bounded in $H^{1}(\mathbb{R}^N)$. Therefore, there is $\omega\in H^{1}(\mathbb{R}^N)$ such that
\begin{equation*}
\omega_{n}\rightharpoonup \omega\quad\text{in}\,\, H^{1}(\mathbb{R}^N)
\end{equation*}
and
\begin{equation*}
\int_{B_R(0)}|\omega|^{2}\,dx=\liminf_{n\to +\infty}\int_{B_R(0)}|\omega_{n}|^{2}\,dx=\liminf_{n\to +\infty}\int_{B_R(z_{n})}|u_{n}|^{2}\,dx\geq \frac{\delta}{2},
\end{equation*}
showing that $\omega\neq 0$.

Now, for each $\phi\in C_{0}^{\infty}(\mathbb{R}^N)$, we have
$$
\int \nabla \omega_{n} \nabla \phi dx +\int V(\epsilon_{n}z_{n}+\epsilon_{n}x)\omega_{n} \phi dx -\int \omega_{n}\phi\log \omega_{n}^{2}dx=o_{n}(1)\Vert \phi\Vert,
$$
which implies that $\omega$ is a nontrivial solution of the problem
\begin{equation}\label{eq34}
-\Delta u+ \alpha_{1}u=u \log u^2  \quad \mbox{in} \quad \mathbb{R}^{N},
\end{equation}
where $\alpha_{1}=\displaystyle \lim_{n\to +\infty}V(\epsilon_{n}z_{n})$. By \cite{squassina}, $\omega\in C^{2}(\mathbb{R}^N)\cap H^{1}(\mathbb{R}^N)$.

For each $k\in N$, there is $\phi_{k}\in C_{0}^{\infty}(\mathbb{R}^N)$ such that
$$
\Vert \phi_{k}-\omega\Vert\rightarrow 0\quad \text{as}\,\, k\rightarrow+\infty,
$$
that is,
$$
\Vert \phi_{k}-\omega\Vert=o_{k}(1).
$$
Using $\frac{\partial \phi_{k}}{\partial x_{i}}$ as a test function, we obtain
$$
\int \nabla \omega_{n} \nabla \frac{\partial \phi_{k}}{\partial x_{i}} dx +\int V(\epsilon_{n}z_{n}+\epsilon_{n}x)\omega_{n} \frac{\partial \phi_{k}}{\partial x_{i}} dx -\int \omega_{n}\frac{\partial \phi_{k}}{\partial x_{i}}\log \omega_{n}^{2}dx=o_{n}(1).
$$
Now, using well known arguments,
$$
\int \nabla \omega_{n} \nabla \frac{\partial \phi_{k}}{\partial x_{i}} dx=\int \nabla \omega \nabla \frac{\partial \phi_{k}}{\partial x_{i}} dx+o_{n}(1)
$$
and
$$
\int \omega_{n}\frac{\partial \phi_{k}}{\partial x_{i}}\log \omega_{n}^{2}dx=\int \omega\frac{\partial \phi_{k}}{\partial x_{i}}\log \omega^{2}dx+o_{n}(1).
$$
Combining the above limit with (\ref{eq34}), we derive that
$$
\limsup_{n\to +\infty}\Big|\int (V(\epsilon_{n}z_{n}+\epsilon_{n}x)-V(\epsilon_{n}z_{n}))\omega_{n} \frac{\partial \phi_{k}}{\partial x_{i}} dx \Big|=0.
$$
As $\phi_{k}$ has compact support, the above limit gives
$$
\limsup_{n\to +\infty}\Big|\int (V(\epsilon_{n}z_{n}+\epsilon_{n}x)-V(\epsilon_{n}z_{n}))\omega \frac{\partial \phi_{k}}{\partial x_{i}} dx \Big|=0.
$$
Now, recalling that $\frac{\partial \omega}{\partial x_{i}}\in L^{2}(\mathbb{R}^N)$, we have that $(\frac{\partial  \phi_{k}}{\partial x_{i}})$
is bounded in $L^{2}(\mathbb{R}^N)$. Hence,
$$
\limsup_{n\to +\infty}\Big|\int (V(\epsilon_{n}z_{n}+\epsilon_{n}x)-V(\epsilon_{n}z_{n}))\phi_{k} \frac{\partial \phi_{k}}{\partial x_{i}} dx \Big|=o_{k}(1),
$$
and so,
$$
\limsup_{n\to +\infty}\Big|\frac{1}{2}\int (V(\epsilon_{n}z_{n}+\epsilon_{n}x)-V(\epsilon_{n}z_{n}))\frac{\partial (\phi_{k}^{2})}{\partial x_{i}} dx \Big|=o_{k}(1).
$$
Using Green's Theorem together with the fact that $\phi_{k}$ has compact support, we have the following limit
$$
\limsup_{n\to +\infty}\Big|\int \frac{\partial V}{\partial x_{i}}(\epsilon_{n}z_{n}+\epsilon_{n}x)\phi_{k}^{2} dx \Big|=o_{k}(1),
$$
which combined with $(V2)$ yields
$$
\limsup_{n\to +\infty}\Big|\frac{\partial V}{\partial x_{i}}(\epsilon_{n}z_{n})\int \vert \phi_{k}\vert^{2} dx \Big|=o_{k}(1).
$$
As
$$
\int \vert \phi_{k}\vert^{2} dx\rightarrow \int \vert \omega\vert^{2} dx\quad \text{as}\,\, k\rightarrow+\infty,
$$
it follows that
$$
\limsup_{n\to +\infty}\Big|\frac{\partial V}{\partial x_{i}}(\epsilon_{n}z_{n}) \Big|=o_{k}(1), \quad \forall \{1, \cdots, N\}.
$$
Since $k$ is arbitrary, we obtain that
$$
\nabla V(\epsilon_{n}z_{n})\rightarrow 0\quad \text{as}\,\, n\rightarrow \infty.
$$
Therefore, $(\epsilon_{n}z_{n})$ is a $(PS)_{\alpha_{1}}$ sequence for $V$, which is a contradiction, because by assumption $V$
satisfies the $(PS)$ condition and $(\epsilon_{n}z_{n})$ does not have any convergent subsequence in $\mathbb{R}^N$.

\end{proof}

The next  lemma will be crucial in our study to show an important estimate, see Lemma \ref{vicente} in Section 3.

\begin{lemma}\label{vice}
Let $\epsilon_{n}\rightarrow 0$ and $(u_{n})\subset \mathcal{N}_{\epsilon_{n}}$ such that $J_{\epsilon_{n}}(u_{n}) \to m(c_{0})$.
Then, there are $(z_{n})\subset \mathbb{R}^N$ with $\vert z_{n}\vert \rightarrow +\infty$ and $u_{1}\in H^{1}(\mathbb{R}^N)\setminus \{0\}$
such that
$$
u_{n}(\cdot+z_{n})\rightarrow u_{1}\quad \text{in}\,\, H^{1}(\mathbb{R}^N).
$$
Moreover, $\displaystyle \liminf_{n\to +\infty}\vert \epsilon_{n}z_{n}\vert>0$.
\end{lemma}

\begin{proof}
Since $u_{n}\in \mathcal{N}_{\epsilon_{n}}$, we have that $J'_{c_{0}}(u_{n})u_{n}\leq 0$ and $J_{c_{0}}(u)\leq J_{\epsilon_{n}}(u)$ for all
$u\in H^{1}(\mathbb{R}^N)$ and $n\in \mathbb{N}$. From this, there is $t_{n}\in (0, 1]$ such that
\begin{equation*}
(t_{n}u_{n})\subset \mathcal{N}_{c_{0}}\quad \text{and}\quad J_{c_{0}}(t_{n}u_{n})\rightarrow m(c_{0}).
\end{equation*}
Since $(t_{n})$ is bounded, by \cite[Section 6]{AlvesdeMorais}, there are $(z_{n})\subset \mathbb{R}^N$, $u_{1}\in  H^{1}(\mathbb{R}^N)\setminus \{0\}$, and a subsequence of $(u_{n})$,
still denote by $(u_{n})$, verifying
$$
u_{n}(\cdot+z_{n})\rightarrow u_{1}\quad \text{in}\,\, H^{1}(\mathbb{R}^N).
$$
Now, we claim that $\displaystyle \liminf_{n\to +\infty}\vert \epsilon_{n}z_{n}\vert>0$. Indeed, as $u_{n}\in \mathcal{N}_{\epsilon_{n}}$ for all $n\in N$,
the function $u_{n}^{1}=u_{n}(\cdot+z_{n})$ must verify
\begin{equation}\label{eq36}
\int(|\nabla u_{n}^{1}|^2+V(\epsilon_{n}z_{n}+\epsilon_{n}x)|u_{n}^{1}|^2)dx+ \int F'_{1}(u_{n}^{1})u_{n}^{1}dx=\int F'_{2}(u_{n}^{1})u_{n}^{1}dx.
\end{equation}
Since $F_{1}$ is convex, even and $F_{1}(t)\geq F_{1}(0)=0$ for all $t\in \mathbb{R}$, we derive that $0\leq F_{1}(t)\leq F'_{1}(t)t$ for all $t\in \mathbb{R}$. Supposing by contradiction that for some subsequence
$$
\lim_{n\to +\infty} \epsilon_{n}z_{n}=0,
$$
taking the limit of $n\to +\infty$ in (\ref{eq36}), we derive the following inequality
\begin{equation*}
\int(|\nabla u_{1}|^2+V(0)|u_{1}|^2)dx+ \int F'_{1}(u_{1})u_{1}dx\leq\int F'_{2}(u_{1})u_{1}dx.
\end{equation*}
Thus, there exists $t_{1}\in (0, 1]$ such that $t_{1}u_{1}\in  \mathcal{N}_{V(0)}$.
Therefore, we have
\begin{equation}\label{eq37}
J_{V(0)}(t_{1}u_{1})\geq m(V(0))>m(c_{0})>0.
\end{equation}
On the other hand,  one has
\begin{eqnarray*}
\lim_{n\to +\infty}J_{\epsilon_{n}}(u_{n})=\lim_{n\to +\infty}\frac{1}{2}\int |u_{n}^{1}|^2 dx&=&\frac{1}{2}\int |u_{1}|^2 dx\\
&\geq &\frac{1}{2}\int |t_{1} u_{1}|^2 dx\\
&=&J_{V(0)}(t_{1}u_{1}),
\end{eqnarray*}
that is,
\begin{equation*}\label{eq38}
m(c_{0})\geq J_{V(0)}(t_{1}u_{1}),
\end{equation*}
which contradicts (\ref{eq37}), finishing the proof.
\end{proof}

\section{A special minimax level}

In order to prove Theorem 1.1, we shall consider a special minimax level. At first, we fix the barycenter function by
$$
\beta(u)=\frac{\int \frac{x}{|x|} |u|^{2}\,dx}{\int |u|^{2}\,dx}, \quad \forall\, u\in H^{1}(\mathbb{R}^N)\setminus \{0\}.
$$
In what follows, $u_{0}$ denotes a positive ground state solution for $J_{c_{0}}$, that is,
$$
J_{c_{0}}(u_{0})=m(c_{0})\quad \text{and}\quad J'_{c_{0}}(u_{0})=0.
$$
Moreover, by \cite{squassina}, $u_{0}$ is also radial. For each $z\in \mathbb{R}^N$ and $\epsilon>0$, we set the function
$$
\phi_{\epsilon, z}(x)=t_{\epsilon, z}u_{0}\Big(x-\frac{z}{\epsilon}\Big),
$$
where $t_{\epsilon, z}>0$ is such that $\phi_{\epsilon, z}\in \mathcal{N}_{\epsilon}$. In what follows, we set $\Phi_{\epsilon}(z)=\phi_{\epsilon, z}$ for all $z \in \mathbb{R}^N$.

\begin{lemma} \label{PhiEP}  The function $\Phi_{\epsilon}:\mathbb{R}^N \to \mathcal{N}_{\epsilon} $ is a continuous function.
	
\end{lemma}
\begin{proof} Let $(z_n) \subset \mathbb{R}^N$ and $z \in  \mathbb{R}^N$ with $z_n\rightarrow z$ in $\mathbb{R}^N$. We must prove that
$$
\Phi_{\epsilon}(z_n) \to \Phi_{\epsilon}(z) \quad \mbox{in} \quad H^{1}(\mathbb{R}^N).
$$	
Here, the main point is to prove that
$$
t_{\epsilon,z_n} \to t_{\epsilon,z} \quad \mbox{in} \quad \mathbb{R}.
$$
By definition of $t_{\epsilon,z_n}$ and $t_{\epsilon,z}$, they are the unique numbers that satisfy
$$
J_{\epsilon}(t_{\epsilon,z_n}u_{0}\Big(\cdot-\frac{z_n}{\epsilon}\Big))=\displaystyle \frac{1}{2} \int |t_{\epsilon,z_n}u_{0}\Big(x-\frac{z_n}{\epsilon}\Big)|^{2}\,dx
$$
and
$$
J_{\epsilon}(t_{\epsilon,z}u_{0}\Big(\cdot-\frac{z}{\epsilon}\Big))=\displaystyle \frac{1}{2} \int |t_{\epsilon,z}u_{0}\Big(x-\frac{z}{\epsilon}\Big)|^{2}\,dx
$$
that is,
\begin{equation} \label{F11}
\frac{1}{2}\int (|t_{\epsilon,z_n}\nabla u_0|^{2}+(V(\epsilon x+z_n)+1)|t_{\epsilon,z_n}u_0|^2)\,dx+\int F_1(t_{\epsilon,z_n}u_0)\,dx-\int F_2(t_{\epsilon,z_n}u_0)\,dx=\frac{1}{2} \int |t_{\epsilon,z_n}u_{0}|^{2}\,dx
\end{equation}
and
\begin{equation} \label{F11*}
\frac{1}{2}\int (|t_{\epsilon,z} \nabla u_0|^{2}+(V(\epsilon x+z)+1)|t_{\epsilon,z} u_0|^2)\,dx+\int F_1(t_{\epsilon,z}u_0)\,dx-\int F_2(t_{\epsilon,z}u_0)\,dx=\frac{1}{2} \int |t_{\epsilon,z}u_{0}|^{2}\,dx.
\end{equation}
A simple calculation gives that $(t_{\epsilon,z_n})$ is bounded, thus for some subsequence, we can assume that $t_{\epsilon,z_n} \to t_*$. Since $F_1$ is increasing in $[0,+\infty)$ and $F_1(lu_0) \in L^{1}(\mathbb{R}^N)$ for all $l  >0$, taking the limit of $n \to +\infty$ in (\ref{F11}) and applying the Lebesgue Theorem, we can conclude that
\begin{equation} \label{F11**}
\frac{1}{2}\int (|t_* \nabla u_0|^{2}+(V(\epsilon x+z)+1)|t_*u_0|^2)\,dx+\int F_1(t_*u_0)\,dx-\int F_2(t_*u_0)\,dx=\frac{1}{2} \int |t_*u_{0}|^{2}\,dx.
\end{equation}	
By uniqueness of $t_{\epsilon,z}$, it follows that $t_{\epsilon,z}=t_*$, and so, $t_{\epsilon,z_n} \to t_{\epsilon,z}$. Now, since
$$
u_{0}\Big(\cdot-\frac{z_n}{\epsilon}\Big) \to u_{0}\Big(\cdot-\frac{z}{\epsilon}\Big) \quad \mbox{in} \quad H^{1}(\mathbb{R}^N),
$$
we get the desired result.
\end{proof}

From definition of $\beta$, we have
the following result.
\begin{lemma}\label{lh}
 For each $r>0$, $\underset{\epsilon\rightarrow 0}{\lim}\Big(\sup\Big\{\Big|\beta(\Phi_{\epsilon}(z))-\frac{z}{\vert z\vert}\Big|: \vert z\vert\geq r\Big\}\Big)=0$.
\end{lemma}

\begin{proof}
The proof follows by showing that for any $(z_{n})\subset \mathbb{R}^N$ with $\vert z_{n}\vert\geq r$ and $\epsilon_{n}\rightarrow 0$, we have that
\begin{equation} \label{NEW1}
\Big|\beta(\Phi_{\epsilon_n}(z_n))-\frac{z_{n}}{\vert z_{n}\vert}\Big|\rightarrow 0\quad \text{as}\,\, n\rightarrow +\infty.
\end{equation}
By change of variables,
$$
\Big|\beta(\Phi_{\epsilon_{n}}(z_{n}))-\frac{z_{n}}{\vert z_{n}\vert}\Big|=\frac{ \displaystyle \int \Big| \frac{\epsilon_{n} x+z_{n}}{\vert \epsilon_{n} x+z_{n}\vert}-\frac{z_{n}}{\vert z_{n}\vert}\Big|\vert u_{0}(x)\vert^{2} dx}{\int \vert u_{0}(x)\vert^{2} dx}.
$$
As for each $x\in \mathbb{R}^N$,
$$
 \Big| \frac{\epsilon_{n} x+z_{n}}{\vert \epsilon_{n} x+z_{n}\vert}-\frac{z_{n}}{\vert z_{n}\vert}\Big|\rightarrow 0\quad\text{as}\quad n\rightarrow +\infty,
$$
by the Lebesgue Dominated Convergence Theorem, we have that
$$
\int \Big| \frac{\epsilon_{n} x+z_{n}}{\vert \epsilon_{n} x+z_{n}\vert}-\frac{z_{n}}{\vert z_{n}\vert}\Big|\vert u_{0}(x)\vert^{2} dx\rightarrow 0\quad\text{as}\quad n\rightarrow +\infty,
$$
proving (\ref{NEW1}).
\end{proof}

As a by-product of  Lemma \ref{lh}, we have the following corollary.
\begin{corollary}\label{corolario2}
Fixed $r>0$, there is $\epsilon_{0}>0$ such that
$$
(\beta(\Phi_{\epsilon}(z)), z)>0, \quad \forall \, \vert z\vert\geq r\quad \text{and}\quad \forall\, \epsilon\in (0, \epsilon_{0}).
$$
\end{corollary}
\begin{proof}
By  last lemma, for fixed $r>0$, there exits $\epsilon_{0}>0$ such that
$$
\Big|\beta(\Phi_{\epsilon}(z))-\frac{z}{\vert z\vert}\Big|<\frac{1}{2}, \quad  \forall \, \vert z\vert\geq r\quad \text{and}\quad \forall\, \epsilon\in (0, \epsilon_{0}).
$$
On the other hand, notice that
$$
(\beta(\Phi_{\epsilon}(z)), z)=(\beta(\Phi_{\epsilon}(z))-\frac{z}{\vert z\vert}, z)+(\frac{z}{\vert z\vert}, z)=(\beta(\Phi_{\epsilon}(z))-\frac{z}{\vert z\vert}, z)+\vert z\vert,\quad \forall z\in \mathbb{R}^N\setminus \{0\}.
$$
Therefore, for $\vert z\vert\geq r$,
$$
(\beta(\Phi_{\epsilon}(z)), z)\geq \vert z\vert\Big(1-\Big|\beta(\Phi_{\epsilon}(z))-\frac{z}{\vert z\vert}\Big|\Big)>\frac{\vert z\vert}{2}\geq \frac{r}{2}>0,
$$
which completes the proof.
\end{proof}
Now, we define the set
$$
\mathcal{B}_{\epsilon}=\{u\in \mathcal{N}_{\epsilon}: \beta(u)\in Y\}.
$$
Since $\beta(\phi_{\epsilon, 0})=0\in Y$, for all $\epsilon>0$, we have that $\mathcal{B}_{\epsilon}\neq \emptyset$. Associated with the above set,
let us consider the real number $D_{\epsilon}$ given by
$$
D_{\epsilon}=\inf_{u\in \mathcal{B}_{\epsilon}}J_{\epsilon}(u).
$$
The next lemma establishes an important relation between the levels $D_{\epsilon}$ and $m(c_{0})$.

\begin{lemma}\label{vicente}
\noindent $(i)$ There exist $\epsilon_{0}, \sigma>0$ such that
$$
D_{\epsilon}\geq m(c_{0})+\sigma,\quad \forall\, \epsilon\in (0, \epsilon_{0}).
$$
\noindent $(ii)$ $\displaystyle \limsup_{\epsilon\to 0}\Big\{\sup_{x\in X}J_{\epsilon}(\Phi_{\epsilon}(x)))\Big\}< 2m(c_{0})-\sigma$.\\
\end{lemma}
\begin{proof} \mbox{} \\
\noindent $(i)$:  From the definition of $D_{\epsilon}$, we know that
$$
D_{\epsilon}\geq m(c_{0})\quad \forall\, \epsilon>0.
$$
Supposing by contradiction that the lemma does not hold, there exists $\epsilon_{n}\rightarrow 0$ satisfying
$$
D_{\epsilon_{n}}\rightarrow m(c_{0}).
$$
Hence, there is $u_{n}\in \mathcal{N}_{\epsilon_{n}}$ with $\beta(u_{n})\in Y$ such that
$$
J_{\epsilon_{n}}(u_{n})\rightarrow m(c_{0}).
$$
Applying Lemma \ref{vice}, there are $u_{1}\in H^{1}(\mathbb{R}^N)\setminus \{0\}$ and a sequence $(z_{n})\subset \mathbb{R}^N$ with  $\displaystyle \liminf_{n\to +\infty}\vert \epsilon_{n}z_{n}\vert>0$ verifying
$$
u_{n}(\cdot+z_{n})\rightarrow u_{1}\quad \text{in}\,\, H^{1}(\mathbb{R}^N),
$$
that is
$$
u_{n}=u_{1}(\cdot-z_{n})+\omega_{n}\quad \text{with}\quad \omega_{n}\rightarrow 0 \quad \text{in}\,\, H^{1}(\mathbb{R}^N).
$$
From the definition of $\beta$,
$$
\beta(u_{1}(\cdot-z_{n}))=\frac{\displaystyle \int \frac{\epsilon_{n} x+\epsilon_{n} z_{n}}{|\epsilon_{n} x+\epsilon_{n}z_{n}|} |u_{1}|^{2}\,dx}{\int |u_{1}|^{2}\,dx}.
$$
Repeating the same arguments explored in the proof of Lemma \ref{lh}, we see that
$$
\beta(u_{1}(\cdot-z_{n}))=\frac{z_{n}}{\vert z_{n}\vert}+o_{n}(1),
$$
and so,
$$
\beta(u_{n})=\beta(u_{1}(\cdot-z_{n}))+o_{n}(1)=\frac{z_{n}}{\vert z_{n}\vert}+o_{n}(1).
$$
Since $\beta(u_{n})\in Y$, we infer that $\frac{z_{n}}{\vert z_{n}\vert}\in Y_{\lambda}$ for $n$ large enough. Consequently,
 $z_{n}\in Y_{\lambda}$ for $n$ large enough, implying that
$$
\liminf_{n\rightarrow\infty}V(\epsilon_{n}z_{n})>c_{0}.
$$
Assuming $A=\displaystyle\liminf_{n\rightarrow\infty}V(\epsilon_{n}z_{n})$, the last inequality together with the Fatou's lemma yields,
$$
m(c_{0})=\liminf_{n\rightarrow\infty}J_{\epsilon_{n}}(u_{n})\geq \liminf_{n\rightarrow\infty}J_{\epsilon_{n}}(t u_{n})\geq J_{A}(t u_{1})\geq m(A)>m(c_{0}),
$$
which is a contradiction. Here $t\in (0, 1]$ such that $J'_{A}(t u_{1})t u_{1}=0$ and $u_{1}\neq 0$.

\noindent $(ii)$: By condition $(V4)$ and the fact that $u_{0}$ is a ground state solution associated with $J_{c_{0}}$, we deduce that
$$
\aligned
\limsup_{\epsilon\to 0}\Big\{\sup_{x\in X}J_{\epsilon}(\Phi_{\epsilon}(x)\Big\}&\leq \dis\frac{1}{2}\int \big (|\nabla u_{0}|^2+(c_{1} +1)|u|^2\big)dx-\dis\frac{1}{2}\int u_{0}^2\log u_{0}^2dx\\
&\leq J_{c_{0}}(u_{0})+\frac{3}{10}c_{2}\int |u_{0}|^2 dx\\
&\leq J_{c_{0}}(u_{0})+\frac{3}{5}J_{c_{0}}(u_{0})=m(c_{0})+\frac{3}{5}m(c_{0})<2m(c_{0}),\,\, \forall\, \epsilon\in (0, \epsilon_{0}),
\endaligned
$$
where $c_{2}=\min\{1, c_{0}\}$.
\end{proof}
Now, we are ready to show the minimax level. In what follows, we fix $\epsilon \in (0, \epsilon_0)$, $\Phi=\Phi_{\epsilon}, J=J_{\epsilon}$ and the following sets
$$
J_{d}=\{u \in H^{1}(\mathbb{R}^N)\,:\,J(u)\leq d\}, \quad Q=\overline{B}_R(0) \cap X \quad \mbox{and} \quad \partial Q= \partial \overline{B}_R(0) \cap X.
$$
Using the above notations, we define the class of the functions
$$
\Gamma=\Big\{h\in C(Q,\, K_{r}): h(x)=\Phi(x), \quad \forall x \in \partial Q  \Big\}
$$
where $\approx$ denotes the homotopy relation, $r>0$, $K=\Phi(Q)$ and
$$
K_r=\{u \in H^{1}(\mathbb{R}^N)\,:\,dist(u,K)<r\}.
$$
Note that $\Gamma \not= \emptyset $, because Lemma \ref{PhiEP} ensures that $\Phi \in \Gamma$.

In what follows we set
$$
\Upsilon_{r}=\{u \in K_r\,:\, \beta(u) \in Y \},
$$
which is not empty because $\mathcal{B}_{\epsilon} \subset K_r$.

\begin{lemma} \label{R} There is $r_0>0$ such that such that
$$
\Theta_r=\inf_{u \in \Upsilon_r}J(u)> m(c_0)+\sigma/2, \quad \forall r \in (0,r_0).
$$
Moreover, there exists $R>0$ such that
$$
J_{\epsilon}(\Phi_{\epsilon}(x))\leq \frac{1}{2}(m(c_{0})+\Theta_r), \quad  \forall\,x\in\partial B_{R}(0)\cap X.
$$
	
\end{lemma}
\begin{proof} Assume by contradiction that the lemma does not hold. Then, there is $r_n \to 0$ and $u_n \in \Upsilon_{r_n}$ such that $J(u_n) \leq m(c_0)+\sigma/2$. From definition of $\Upsilon_{r_n}$, there is $v_n \in K$ such that $\|u_n-v_n\| \leq r_n$. Since $K$ is compact, there is a subsequence of $(v_n)$, still denoted by itself, and $v \in K$ such that $v_n \to v$ in $H^{1}(\mathbb{R}^N)$, then $u_n \to v$ in $H^{1}(\mathbb{R}^N)$ and $\beta(v) \in Y$, from where it follows that $v \in \mathcal{B}_{\epsilon}$, then by Lemma \ref{vicente}(i),  $J(v) \geq m(c_0)+\sigma$. On the other hand, since $J_{\epsilon}$ is l.s.c., we have that $\displaystyle \liminf_{n\to +\infty}J(u_n) \geq J(v)$, which is absurd.

By $(V1)$, given $\delta>0$, there are $\epsilon_{0}, R>0$ such that
$$
\sup\Big\{J_{\epsilon}(\Phi_{\epsilon}(x)): x\in \partial B_{R}(0)\cap X\Big\}\leq m(c_{0})+\delta,\,\, \forall\, \epsilon\in (0, \epsilon_{0}).
$$
Fixing $\delta=\frac{\sigma}{4}$, where $\sigma$ was given in $(i)$, we have that
$$
\sup\Big\{J_{\epsilon}(\Phi_{\epsilon}(x)): x\in \partial B_{R}(0)\cap X\Big\}\leq \frac{1}{2}\Big(2m(c_{0})+\frac{\delta}{2}\Big)<\frac{1}{2}(m(c_{0})+\Theta_{r}),\,\, \forall\, \epsilon\in (0, \epsilon_{0}).
$$

\end{proof}

\begin{lemma}\label{te}
If $h \in \Gamma$, then $h(Q) \cap \Upsilon_{r}\neq \emptyset$ for all $r \in (0,r_0).$
\end{lemma}

\begin{proof}
It is enough to show that for all $h\in \Gamma$, there is $x_{*}\in Q$ such that
$$
\beta(h(x_{*}))\in Y.
$$
For each $h\in \Gamma$, we set the function $g: Q \rightarrow \mathbb{R}^N$ given by
$$
g(x)=\beta(h(x))\quad \forall x\in Q,
$$
and the homotopy $\mathcal{F}: [0, 1]\times Q\rightarrow X$
as
$$
\mathcal{F}(t, x)=tP_{X}(g(x))+(1-t)x,
$$
where $P_{X}$ is the projection onto $X=\{(x, 0): x\in \mathbb{R}^N\}$. By Corollary \ref{corolario2}, fixed $R>0$ and $\epsilon>0$ small enough, we have that
$$
(\mathcal{F}(t, x), x)>0, \quad \forall (t, x)\in [0, 1]\times \partial Q .
$$
Using the homotopy invariance property of the Topological degree, we derive
$$
d(g, Q, 0)=1,
$$
implying that there exists $x_{*}\in Q$ such that $\beta(h(x_{*}))=0$.
\end{proof}

Now, define the minimax value
$$
C_{\epsilon}=\inf_{h \in \Gamma}\sup_{x \in Q}J(h(x)).
$$
From Lemmas \ref{R} and \ref{te},
\begin{equation}\label{eq41}
C_{\epsilon}\geq \Theta_r= \inf_{u \in \Upsilon_r}J(u) \geq m(c_{0})+\sigma/2.
\end{equation}
On the other hand,
$$
C_{\epsilon}\leq \sup_{x\in Q}J(\Phi(x)).
$$
Then, by Lemma \ref{vicente}(ii),
\begin{equation}\label{eq42}
C_{\epsilon}\leq \sup_{x\in Q}J(\Phi(x))<2m(c_{0})-\sigma.
\end{equation}
From (\ref{eq41})  and (\ref{eq42}),
$$
C_{\epsilon}\in (m(c_{0})+\sigma/2, \,\, 2m(c_{0})-\sigma).
$$
Now, by Lemma \ref{le22}, we know that the $J$ satisfies the $(PS)_{c}$ condition for all $c\in (m(c_{0})+\sigma/2, \,\, 2m(c_{0})-\sigma)$,
hence $J$ satisfies the $(PS)_{C_{\epsilon}}$ condition for $\epsilon$ small enough. Using this fact we can ensure that $C_{\epsilon}$
is a critical level for $J$. To see why, we will follow the same type of ideas found in the proof of \cite[Theorem 3.4]{szulkin}. Have this in mind, by Lemma \ref{R}, we can fix  for $\tau >0$ small enough such that
$$
C_\epsilon - \tau/2 > \frac{1}{2}(m(c_0)+\Theta_r),
$$
and we set
$$
\Gamma_1=\Big\{h\in C(Q,\, K_r): h|_{\partial Q} \approx \Phi|_{\partial Q} \quad \mbox{in} \quad J_{C_\epsilon - \tau/4}, \quad \sup_{x \in \partial Q}J(h(x))\leq C_\epsilon - \tau/2 \Big\}
$$
where $\approx$ denotes the homotopy relation  and the number
$$
C^*=\inf_{h \in \Gamma_1}\sup_{x \in Q}J(h(x)).
$$
Arguing as in \cite[Theorem 3.4]{szulkin} we have that $C^*=C_\epsilon$, and so, it is enough to prove that $C^*$ is a critical level. In order to show this, we argue by contradiction by supposing that $C^*$ is not a critical point and fixing $\tau>0$ small enough and $h \in \Gamma_1$ such that
\begin{equation} \label{D0}
\Pi(h) \leq C^*+\tau \quad \mbox{and} \quad \Pi(g)-\Pi(h) \geq -\tau d(g,h), \quad \forall g \in \Gamma_1
\end{equation}
where
$$
\Pi(g)=\sup_{x \in Q}J(g(x)), \quad \forall g \in \Gamma_1
$$
and
$$
d(g,h)=\sup_{x\in Q}\|g(x)-h(x)\|.
$$
Now, we apply \cite[Proposition 2,3]{szulkin} with $A=h(Q)$ to find a closed subset $W$ containing $A$ in it interior and a deformation $\alpha_s:W \to H^{1}(\mathbb{R}^N)$ having the following  properties: \\
\begin{equation} \label{D1}
\|u-\alpha_s(u)\| \leq s, \quad \forall u \in W \quad \mbox{and} \quad s \approx 0^+,
\end{equation}
\begin{equation} \label{D2}
J(\alpha_s(u))-J(u)\leq 2s, \quad \forall u \in W,
\end{equation}
\begin{equation} \label{D3}
J(\alpha_s(u))-J(u) \leq -2\tau s \quad \forall u \in W \quad \mbox{with} \quad J(u) \geq C^*-\tau
\end{equation}
and
\begin{equation} \label{D4}
\sup_{u\in A}J(\alpha_s(u))-\sup_{u\in A}J(u)\leq -2\tau s.
\end{equation}
Now, it is easy to see that $g=\alpha_s \circ h \in \Gamma_1$ for $s$ small enough. However, from (\ref{D0}), (\ref{D1}) and (\ref{D4})
$$
-2\tau s \geq \Pi(g) - \Pi(h) \geq -\tau d(g,h) \geq -\tau s,
$$
which is a contradiction. This contradiction shows that $C^*$ is a critical level and the proof is completed. The fact that $C^* \in (m(c_0),2m(c_0))$ permits to conclude that the solutions with energy equal to $C^*$ do not change sign, and as $f(t)=t\log t^2$ is an odd function, we can assume that they are nonnegative. Now, the positivity follows by maximum principle.

\vspace{1 cm}

\noindent \textsc{Claudianor O. Alves } \\
Unidade Acad\^{e}mica de Matem\'atica\\
Universidade Federal de Campina Grande \\
Campina Grande, PB, CEP:58429-900, Brazil \\
\texttt{coalves@mat.ufcg.edu.br} \\
\noindent and \\
\noindent \textsc{Chao Ji}(Corresponding Author) \\
Department of Mathematics\\
East China University of Science and Technology \\
Shanghai 200237, PR China \\
\texttt{jichao@ecust.edu.cn}

\end{document}